\documentclass[11pt]{article}
\usepackage{amsmath}
\usepackage{amssymb}
\usepackage{amsthm}
\begin{document}
\def\N{{\mathbb N}}
\def\Z{{\mathbb Z}}
\def\R{{\mathbb R}}
\def\C{{\mathbb C}}
\def\T{{\mathbb T}}
\def\K{{\mathbb K}}
\def\U{\overline{U}}
\def\D{{\mathbb D}}
\def\epsilon{\varepsilon}
\def\phi{\varphi}
\def\kappa{\varkappa}
\def\leq{\leqslant}
\def\geq{\geqslant}
\def\RR{{\cal R}}
\def\Re{{\tt Re}\,}
\def\ssub#1#2{#1_{{}_{{\scriptstyle #2}}}}
\def\li{\,\text{\rm span}\,}
\def\Y{{\cal Y}}
\def\slim{\mathop{\hbox{$\overline{\hbox{\rm lim}}$}}}
\def\ilim{\mathop{\hbox{$\underline{\hbox{\rm lim}}$}}}
\def\X{{\cal X}}
\def\liu{\vrule width0pt height14pt depth0pt}
\def\lid{\vrule width0pt height0pt depth10pt}
\font\CC=cmti10 scaled \magstep{2}
\def\e{\hbox{\CC e}}

\newtheorem{theorem}{Theorem}[section]
\newtheorem{lemma}[theorem]{Lemma}
\newtheorem{corollary}[theorem]{Corollary}
\newtheorem{proposition}[theorem]{Proposition}
\newtheorem{question}[theorem]{Question}
\newtheorem{definition}[theorem]{Definition}
\newtheorem{example}[theorem]{Example}
\newtheorem{remark}[theorem]{Remark}

\newtheorem*{thmM}{Theorem M}
\newtheorem*{thmMb}{Theorem Mb}
\newtheorem*{thmS}{Theorem S}

\numberwithin{equation}{section}
\renewcommand{\theenumi}{\roman{enumi}}

\title{Pointwise universal trigonometric series}

\author{S.~Shkarin}

\date{}

\maketitle

\smallskip

\rm\normalsize

\begin{abstract} A series $S_a=\sum\limits_{n=-\infty}^\infty a_nz^n$ is
called a {\it pointwise universal trigonometric series} if for any
$f\in C(\T)$, there exists a strictly increasing sequence
$\{n_k\}_{k\in\N}$ of positive integers such that
$\sum\limits_{j=-n_k}^{n_k} a_jz^j$ converges to $f(z)$ pointwise on
$\T$. We find growth conditions on coefficients allowing and
forbidding the existence of a pointwise universal trigonometric
series. For instance, if $|a_n|=O(\e^{\,|n|\ln^{-1-\epsilon}\!|n|})$
as $|n|\to\infty$ for some $\epsilon>0$, then the series $S_a$ can
not be pointwise universal. On the other hand, there exists a
pointwise universal trigonometric series $S_a$ with
$|a_n|=O(\e^{\,|n|\ln^{-1}\!|n|})$ as $|n|\to\infty$.

\bigskip

\noindent{\bf Keywords:} \ Universal series, trigonometric series,
power series

\smallskip

\noindent{\bf MSC:} \ 30B30, 32A05

\end{abstract}

\section{Introduction}

Throughout this article $\C$ is the field of complex numbers,
$\T=\{z\in\C:|z|=1\}$, $\R$ is the field of real numbers, $\Z$ is
the set of integers, $\Z_+$ is the set of non-negative integers and
$\N$ is the set of positive integers. For $z\in\C$, $\Re z$ is the
real part of $z$. For a compact metric space $K$, $C(K)$ stands for
the space of continuous complex-valued functions on $K$. An {\it
interval} in $\T$ is a closed connected subset of $\T$ of positive
length. Symbol $\mu$ stands for the normalized Lebesgue measure on
$\T$. Recall that {\it Fourier coefficients} of $f\in L_1(\T)$ or of
a finite Borel $\sigma$-additive complex valued measure $\nu$ on
$\T$ are given by the formula
$$
\widehat f(n)=\int_\T f(z)z^{-n}\,\mu(dz)\ \ \text{and}\ \ \widehat
\nu(n)=\int_\T z^{-n}\,\nu(dz)\ \ \text{for}\ \ n\in\Z.
$$
In particular, $\widehat f(n)=\widehat \nu(n)$ if $\nu$ is the
measure with the density $f$ with respect to $\mu$. Recall also that
$\{z^n\}_{n\in\Z}$ is an orthonormal basis in $L_2(\T)=L_2(\T,\mu)$,
which yields
\begin{equation}\label{onb}
\langle f,g\rangle_{L_2(\T)}=\int_\T f(z)\overline{g(z)}\,\mu(dz)=
\sum_{n=-\infty}^\infty \widehat f(n)\overline{\widehat g(n)}\ \
\text{for any $f,g\in L_2(\T)$.}
\end{equation}

\begin{definition}\label{d1} \rm Let $X$ be a topological vector space
and let $\{x_n\}_{n\in\Z_+}$ be a sequence in $X$. We say  that
$\sum\limits_{n=0}^\infty x_n$ is a {\it universal series} if for
every $x\in X$, there exists a strictly increasing sequence
$\{n_k\}_{k\in\Z_+}$ of positive integers such that
$\smash{\sum\limits_{j=0}^{n_k} x_j\to x}$ as $k\to\infty$.
Similarly if $\{x_n\}_{n\in\Z}$ is a bilateral sequence in $X$, then
we say that $\sum\limits_{n=-\infty}^\infty x_n$ is a {\it universal
series} if for every $x\in X$, there exists a strictly increasing
sequence $\{n_k\}_{k\in\Z_+}$ of positive integers such that
$\smash{\sum\limits_{j=-n_k}^{n_k} x_j\to x}$ as $k\to\infty$.
\end{definition}

\begin{remark}\label{rr}\rm Note that if $X$ is metrizable, then universality of
$\smash{\sum\limits_{n=0}^\infty x_n}\lid$ $\bigl($respectively, of
$\smash{\sum\limits_{n=-\infty}^\infty x_n}\bigr)$ is equivalent to
density of $\{S_n:n\in\Z_+\}$ in $X$, where
$S_n=\sum\limits_{j=0}^{n} x_j$ $\bigl($respectively,
$S_n=\sum\limits_{j=-n}^{n} x_j\bigr)$. If $X$ is non-metrizable,
the latter may fail due to the fact that a closure of a set can
differ from its sequential closure.
\end{remark}

Universal series have been studied by many authors. Those interested
in the subject ought to look into \cite{u-s}, which together with a
systematic approach to the theory of universal series provides a
large survey part and extensive list of references on the subject.
We would also like to mention papers \cite{u13,u2,u6,u1,u7,u3,u5},
which are not in the list of references in \cite{u-s} and the papers
\cite{tri1,tri4,tri3,tri2,tri5} dealing specifically with universal
trigonometric series. We start by mentioning two old results on
universal series. Seleznev \cite{sel,u-s} proved the following
theorem on universal power series.

\begin{thmS}There exists a sequence $\{a_n\}_{n\in\Z_+}$ of complex
numbers such that for each entire function $f$ and any compact
subset $K$ of $\C\setminus\{0\}$ with connected $\C\setminus K$,
there is a strictly increasing sequence $\{n_k\}_{k\in\Z_+}$ of
positive integers for which $\sum\limits_{j=0}^{n_k} a_jz^j$
converges to $f(z)$ uniformly on $K$.
\end{thmS}

Menshov \cite{men,u-s} constructed a sequence $c\in c_0(\Z)$ such
that the series $\sum\limits_{n=-\infty}^\infty c_nz^n$ is universal
in the space $L_0(\T)$ of (equivalence classes of) $\mu$-measurable
functions $f:\T\to\C$ with the measure convergence topology. Such
series are referred to as universal trigonometric series. Since
$L_0(\T)$ is metrizable, $C(\T)$ is dense in $L_0(\T)$ and every
measure convergent sequence of functions has a subsequence, which is
almost everywhere convergent, the result of Menshov is equivalent to
the following statement.

\begin{thmM}There exists $a\in c_0(\Z)$ such that for each $f\in
C(\T)$, there is a strictly increasing sequence $\{n_k\}_{k\in\Z_+}$
of positive integers for which $\smash{\sum\limits_{j=-n_k}^{n_k}
a_jz^j}\liu$ converges to $f(z)$ for almost all $z\in\T$ with
respect to the Lebesgue measure.
\end{thmM}

We study the natural question whether one can replace almost
everywhere convergence by pointwise convergence. In other words, we
deal with universal trigonometric series in the space $C_p(\T)$,
being $C(\T)$ endowed with the pointwise convergence topology. Our
first observation is an easy consequence of Theorem~S.

\begin{proposition}\label{t1} There exists a sequence
$\{a_n\}_{n\in\Z_+}$ in $\C$ such that for each $f\in C(\T)$, there
is a strictly increasing sequence $\{n_k\}_{k\in\Z_+}$ of positive
integers for which $\sum\limits_{j=0}^{n_k} a_jz^j$ converges to
$f(z)$ for all $z\in\T$.
\end{proposition}

Thus there is a universal power series in $C_p(\T)$. The next
natural step is to try to figure out whether one can make
coefficients of a universal power or at least a trigonometric series
in $C_p(\T)$ small. It turns out in considerable contrast with
Theorem~M that passing from almost everywhere convergence to
pointwise convergence forces the coefficients of a universal
trigonometric series to grow almost exponentially. In order to
formulate the explicit result, we introduce the following notation.

\begin{definition}\label{d2} \rm We say that a sequence $\{c_n\}_{n\in\Z_+}$
of positive real numbers belongs to the class $\X_-$ if
$$
\lim_{n\to\infty} n(c_{n+1}-c_n)=+\infty\ \ \ \text{and}\ \ \
\sum_{n=0}^\infty \frac{c_n}{n^2+1}<\infty.
$$
We say that a sequence $\{c_n\}_{n\in\Z_+}$ of positive real numbers
belongs to the class $\X_+$ if
$$
\lim_{n\to\infty} n(c_{n+1}-c_n)=+\infty\ \ \ \text{and}\ \ \
\sum_{n=0}^\infty \frac{c_n}{n^2+1}=\infty.
$$
\end{definition}

\begin{remark}\label{r1} \rm It is easy to see that for any $\epsilon>0$,
the sequence $\{(n+1)\ln^{-1-\epsilon}(n+2)\}_{n\in\Z_+}$ belongs to
$\X_-$. On the other hand, the sequence
$\{(n+1)\ln^{-1}(n+2)\}_{n\in\Z_+}$ belongs to $\X_+$.
\end{remark}

For a bilateral sequence $a=\{a_n\}_{n\in\Z}$ of complex numbers and
$m\in\Z_+$ we denote
$$
S^a_m(z)=\sum_{k=-m}^m a_kz^k.
$$
In other words, $S^a_m:\T\to \C$ are partial sums of the
trigonometric series with coefficients $a_n$. For an interval $J$ in
$\T$, symbol $\RR_a(J)$  stands for the set of $g\in C(J)$ such that
there is a strictly increasing sequence $\{m_k\}_{k\in\Z_+}$ of
positive integers for which $S^a_{m_k}(z)\to g(z)$ for any $z\in J$.

\begin{theorem}\label{t2} Let $\{a_n\}_{n\in\Z}$ be a bilateral
sequence of complex numbers such that there exists $c\in\X_-$ for
which $|a_{n}|\leq \e^{\,c_{|n|}}$ for any $n\in\Z$. Then for any
interval $J$ in $\T$, the set $\RR_a(J)$ contains at most one
element. In particular, the trigonometric series
$\smash{\sum\limits_{n=-\infty}^\infty a_nz^n}\liu$ is not universal
in $C_p(\T)$.
\end{theorem}

The following theorem shows that Theorem~\ref{t2} is sharp.

\begin{theorem}\label{t3} Let $c\in\X_+$. Then there exists a
sequence $a=\{a_n\}_{n\in\Z}$ of complex numbers such that
$|a_n|\leq \e^{\,c_{|n|}}$ for each $n\in\Z$ and the trigonometric
series $\sum\limits_{n=-\infty}^\infty a_nz^n$ is universal in
$C_p(\T)$. That is, $\RR_a(\T)=C(\T)$.
\end{theorem}

Theorems~\ref{t2}, \ref{t3} and Remark~\ref{r1} imply the following
corollary.

\begin{corollary}\label{c1} There exists a sequence
$\{a_n\}_{n\in\Z}$ such that the trigonometric series
$\sum\limits_{n=-\infty}^\infty a_nz^n$ is pointwise universal and
$a_n=O(\e^{\,|n|\ln^{-1}\!|n|})$. On the other hand, for each
$\epsilon>0$ and any sequence $\{a_n\}_{n\in\Z}$ satisfying
$a_n=O(\e^{\,|n|\ln^{-1-\epsilon}\!|n|})$ for some $\epsilon>0$, the
trigonometric series $\sum\limits_{n=-\infty}^\infty a_nz^n$ is not
pointwise universal.
\end{corollary}

The following question remains open.

\begin{question}\label{que1} Does the conclusion of
Theorem~$\ref{t3}$ remain true under the extra condition that
$a_n=0$ for $n<0$?
\end{question}

\section{Proof of Proposition~\ref{t1}}

Let $a=\{a_n\}_{n\in\Z_+}$ be the sequence provided by Theorem~S. We
shall simply show that it satisfies all requirements of
Proposition~\ref{t1}. Let $f\in C(\T)$ and $J_k=\{e^{it}:0\leq t\leq
2\pi-k^{-1}\}$ for $k\in\N$. Since each $J_k$ is compact, has empty
interior in $\C$ and has connected complement in $\C$, we can use
the classical Mergelyan theorem to pick a sequence
$\{p_k\}_{k\in\N}$ of polynomials such that $|f(z)-p_k(z)|<2^{-k-1}$
for any $k\in\N$ and $z\in J_k$. By Theorem~S, we can choose a
strictly increasing sequence $\{n_k\}_{k\in\N}$ of positive integers
such that $|p_k(z)-S_k(z)|<2^{-k-1}$ for any $k\in\N$ and $z\in
J_k$, where $S_k(z)=\sum\limits_{j=0}^{n_k} a_jz^j$. Hence
$|f(z)-S_k(z)|<2^{-k}$ for any $k\in\N$ and $z\in J_k$. Since each
$z\in\T$ belongs to $J_k$ for each sufficiently large $k$, we see
that $S_k(z)\to f(z)$ as $k\to\infty$ for any $z\in\T$. The proof of
Proposition~\ref{t1} is complete.

\section{Proof of Theorem~\ref{t2}}

We need the following result due to Mandelbrojt \cite{man1}, which
is closely related to the Denjoy--Carleman theorem characterizing
quasi-analytic classes of infinitely differentiable functions. A
formally slightly weaker statement can be found in an earlier work
of Mandelbrojt \cite{man2}, although one can easily see that this
weaker form is actually equivalent to the stronger one. See also
\cite{rudin} for a different proof.

\begin{thmMb} {\rm I.} \ If $c\in \X_-$,
then for any interval $J$ of $\T$ there exists a non-zero infinitely
differentiable function $f:\T\to [0,\infty)$ such that $f(z)=0$ for
$z\in\T\setminus J$ and $|\widehat f(n)|\leq \e^{-c_{|n|}}$ for each
$n\in\Z$.

{\rm II.} \ If $c\in \X_+$, $f:\T\to\C$ is infinitely
differentiable, $|\widehat f(n)|\leq \e^{-c_{|n|}}$ for each
$n\in\Z$ and there is $z_0\in\T$ such that $f^{(j)}(z_0)=0$ for each
$j\in\Z_+$, then $f$ is identically $0$.
\end{thmMb}

We are ready to prove Theorem~\ref{t2}. Assume the contrary. Then
there exists an interval $J$ in $\T$, two different functions
$f,g\in C(J)$ and two strictly increasing sequences
$\{m_n\}_{n\in\Z_+}$ and $\{k_n\}_{n\in\Z_+}$ of positive integers
such that $S^a_{k_n}(z)$ converges $f(z)$ and $S^a_{m_n}(z)$
converges $g(z)$ as $n\to\infty$ for each $z\in J$. Passing to
subsequences, if necessary, we can without loss of generality assume
that $m_n<k_n<m_{n+1}$ for each $n\in\Z_+$. Since $f\neq g$, there
is $z_0\in J$ such that $f(z_0)\neq g(z_0)$. Denote
$p_n=C(S^a_{k_n}-S^a_{m_n})$, where $C=2(f(z_0)-g(z_0))^{-1}$. Then
$p_n(z)$ converges to $h(z)$ for each $z\in J$, where
$h(z)=C(f(z)-g(z))$. Clearly $h\in C(J)$ and $h(z_0)=2$. Since $h$
is continuous, we can pick an interval $J_0\subseteq J$ such that
$\Re h(z)\geq 1$ for each $z\in J_0$. We shall prove the following
statement:
\begin{equation}\label{Q}
\text{for any interval $I\subset J_0$ there is $n_0\in\N$ such that
$\min\limits_{z\in I}\Re p_n(z)<0$ for any $n\geq n_0$}.
\end{equation}

Suppose that (\ref{Q}) is not satisfied. Then there exist an
interval $I\subset J_0$ and a strictly increasing sequence
$\{j_n\}_{n\in\Z_+}$ of positive integers satisfying
$\min\limits_{z\in I}\Re p_{j_n}(z)\geq 0$ for any $n\in\Z_+$. Take
a sequence $c\in \X_-$ such that $|a_n|\leq \e^{\,c_{|n|}}$ for any
$n\in\Z$. It is easy to see that the sequence $\{d_n\}_{n\in\Z_+}$
also belongs $\X_-$, where $d_n=c_n+\ln(n^2+1)$. By Theorem~Mb there
is a non-zero infinitely differentiable function $\rho:\T\to
[0,\infty)$ such that $\rho(z)=0$ for $z\in\T\setminus I$ and
$|\widehat \rho(n)|\leq \e^{-d_{|n|}}=(n^2+1)^{-1}\e^{-c_{|n|}}$ for
each $n\in\Z$. Since $\rho$ vanishes outside $I$, $\Re p_{j_n}\geq
0$ on $I$ and $p_{j_n}$ converge to $h$ pointwise on $J\supseteq I$,
we see that $\rho\,\Re p_{j_n}$ is a sequence of non-negative
continuous functions pointwise convergent to $\rho\,\Re h$. By the
Fatou Theorem
\begin{equation}\label{01}
\ilim_{n\to\infty}\int\limits_\T \rho(z)\,\Re
p_{j_n}(z)\,\mu(dz)\geq \int\limits_\T \rho(z)\,\Re
h(z)\,\mu(dz)\geq \int\limits_I \rho(z)\,\mu(dz)>0
\end{equation}
since $\Re h(z)\geq 1$ for each $z\in J_0\supseteq I$. On the other
hand, using (\ref{onb}), we obtain
\begin{align*}
&\int_\T \rho(z)\Re p_{n}(z)\,\mu(dz)=\Re
\langle\rho,p_{n}\rangle_{L_2(\T)}\leq
|\langle\rho,p_{n}\rangle_{L_2(\T)}|=\biggl|
\sum_{l\in\Z}\widehat{p_{n}}(l)\overline{\widehat \rho(l)}\biggr|=
\biggl|\sum_{m_{n}<|l|\leq k_{n}}\!\!\!\! Ca_l\overline{\widehat
\rho(l)}\biggr|\leq
\\
&\qquad\qquad\leq \sum_{m_{n}<|l|\leq
k_{n}}\frac{|C||a_l|\e^{-c_{|l|}}}{l^2+1}\leq \sum_{m_{n}<|l|\leq
k_{n}} \frac{|C|}{l^2+1}=O(m_{n}^{-1})=o(1)\ \ \text{as
$n\to\infty$.}
\end{align*}
The above display contradicts (\ref{01}), which completes the proof
of (\ref{Q}).

Using (\ref{Q}) and continuity of $p_n$, we can pick  a sequence
$\{J_n\}_{n\in\Z_+}$ of intervals in $\T$ starting from $J_0$ and a
strictly increasing sequence $\{r_n\}_{n\in\N}$ of positive integers
such that $J_{n+1}\subseteq J_n$ for each $n\in\Z_+$ and $\Re
p_{r_n}(z)<0$ for each $z\in J_n$, $n\in\N$. Since any decreasing
sequence of compact sets has nonempty intersection, there is $w\in
\bigcap\limits_{n=0}^\infty J_n$. Then $\Re p_{r_n}(w)<0$ for any
$n\in\N$. Since $p_n(w)\to h(w)$, we have $\Re h(w)\leq 0$. On the
other hand, $\Re h(w)\geq 1$ since $w\in J_0$. This contradiction
completes the proof of Theorem~\ref{t2}.

\section{Proof of Theorem~\ref{t3}}

First, we introduce some notation. Symbol $\omega(\Z)$ stands for
the space $\C^\Z$ of all complex bilateral sequences with the
coordinatewise convergence topology, while $\phi(\Z)$ is the
subspace of $\omega(\Z)$ of sequences with finite support. That is,
$x\in\omega(\Z)$ belongs to $\phi(\Z)$ if and only if there is
$n\in\N$ such that $x_j=0$ if $|j|>n$. Let also $\{e_n\}_{n\in\Z}$
be the canonical linear basis of $\phi(\Z)$. The next result is a
part of Theorem~1 from \cite{u-s}, formulated in an equivalent form
and adapted to sequences labeled by $\Z$ rather than $\Z_+$. For
sake of completeness we sketch its proof.

\begin{lemma}\label{coco} Let $X$ be a metrizable topological vector
space, $A$ be a linear subspace of $\omega(\Z)$ endowed with its own
topology, which turns $A$ into a complete metrizable topological
vector space such that the topology of $A$ is stronger than the one
inherited from $\omega(\Z)$ and $\phi(\Z)$ is a dense linear
subspace of $A$. Assume also that $S:\phi(\Z)\to X$ is a linear map
and $S(U\cap \phi(\Z))$ is dense in $X$ for any neighborhood $U$ of
zero in $A$. Then there exists $a\in A$ such that the series
$\sum\limits_{n=-\infty}^\infty a_nSe_n$ is universal in $X$.
\end{lemma}

\begin{proof} For each $n\in\N$, let $P_n:A\to A$ be the linear map
defined by the formula $P_n a=\sum\limits_{j=-n}^n a_je_j$. Since
the topology of $A$ is stronger than the one inherited from
$\omega(\Z)$, $P_n$ are continuous linear projections on $A$ and
$SP_n$ are continuous linear operators from $A$ to $X$. We use
symbol ${\cal U}$ to denote the set of $a\in A$ such that the series
$\sum\limits_{n=-\infty}^\infty a_nSe_n$ is universal in $X$. Since
the space $S(\phi(\Z))$ of countable algebraic dimension is dense in
$X$, $X$ is separable. Since $X$ is also metrizable, there is a
countable base $\{W_k:k\in\N\}$ of topology of $X$. It is
straightforward to see that
\begin{equation}\label{formula}
{\cal U}=\bigl\{a\in A:\{SP_na:n\in\Z_+\}\ \ \text{is dense in
$X$}\bigr\}=\bigcap_{k=1}^\infty \Omega_k,\ \ \text{where}\ \
\Omega_k=\bigcup_{n=1}^\infty (SP_n)^{-1}(W_k).
\end{equation}
Since the operators $SP_n:A\to X$ are continuous, the sets
$\Omega_k$ are open in $A$.

Linearity of $S$ and density of $S(\phi(\Z)\cap U)$ in $X$ for each
neighborhood of 0 in $A$ imply that $S(\phi(\Z)\cap U)$ is dense in
$A$ for any open $U\subseteq A$ such that $U\cap
\phi(\Z)\neq\varnothing$. Since $\phi(\Z)$ is dense in $A$, we see
that $S(U\cap \phi(\Z))$ is dense in $X$ for any non-empty open
subset $U$ of $A$. Hence for each $k\in\N$ and a non-empty open
subset $U$ of $A$, there is $a\in \phi(\Z)\cap U$ such that $Sa\in
W_k$. Since $a\in\phi(\Z)$, $P_na=a$ for each sufficiently large
$n\in\N$. It follows that $a\in \Omega_k$. Thus $\Omega_k\cap
U\neq\varnothing$ for any non-empty open subset $U$ of $A$. That is,
$\Omega_k$ is a dense open subset of $A$ for each $k\in\N$. Since
$A$ is complete and metrizable, the Baire theorem and
(\ref{formula}) imply that $\cal U$ is a dense $G_\delta$-set in
$A$. In particular, ${\cal U}\neq\varnothing$ as required.
\end{proof}

We are ready to prove Theorem~\ref{t3}. For $k\in\N$ let
\begin{equation}\label{pk}
\text{$J_k=\{e^{it}:0\leq t\leq 2\pi-k^{-1}\}$ \ \ and \ \
$p_k:C(\T)\to\R$, \ $p_k(f)=\sup\limits_{z\in J_k}|f(z)|$.}
\end{equation}
Consider the topological vector space $C_\tau(\T)$ being $C(\T)$
with the topology $\tau$ defined by the sequence $\{p_k\}_{k\in\N}$
of seminorms. Then $C_\tau(\T)$ is metrizable and locally convex.
Moreover, a sequence $\{f_n\}_{n\in\N}$ converges to $f$ in
$C_\tau(\T)$ if and only if $f_n$ converges to $f$ uniformly on
$J_k$ for each $k\in\N$. Since the union of $J_k$ is $\T$,
$\tau$-convergence implies pointwise convergence. Hence $\tau$ is
stronger than the topology of $C_p(\T)$. Thus any series universal
in $C_\tau(\T)$ is also universal in $C_p(\T)$. Let $c\in\X_+$.
Consider the space $A_c$ of complex bilateral sequences
$a=\{a_n\}_{n\in\Z}$ such that $a_n=o\bigl(\e^{\,c_{|n|}}\bigr)$ as
$n\to\infty$. The space $A_c$ with the norm
$\|a\|=\sup\limits_{n\in\Z}|a_n|\e^{-c_{|n|}}$ is a Banach space
isometric to $c_0$. It is also straightforward to see that
$\phi(\Z)$ is dense in $A_c$ and that the topology of $A_c$ is
stronger than the one inherited from $\omega(\Z)$. In order to prove
Theorem~\ref{t3} it is enough to find $a\in A_c$ such that the
trigonometric series $\sum\limits_{n=-\infty}^\infty a_nz^n$ is
universal in $C_\tau(\T)$. Consider the linear map $S:\phi(\Z)\to
C(\T)$ such that $Se_n(z)=z^n$ for $n\in\Z$ and $z\in\T$. By
Lemma~\ref{coco}, it suffices to demonstrate  that for any
neighborhood $U$ of zero in $A_c$, $S(U\cap \phi(\Z))$ is dense in
$C_\tau(\T)$. Since $A_c$ is a normed space, it is enough to show
that $\Omega=S(U\cap \phi(\Z))$ is dense in $C_\tau(\T)$, where
$U=\{a\in A_c:\|a\|\leq 1\}$. Assume the contrary. Then $\Omega$ is
not dense in $C_\tau(\T)$. Since $\Omega$ is convex and balanced
(=stable under multiplication by $z\in\C$ with $|z|\leq 1$) and the
topological vector space $C_\tau(\T)$ is locally convex, the
Hahn--Banach theorem implies that there exists a non-zero continuous
linear functional $F:C_\tau(\T)\to\C$ such that $|F(f)|\leq 1$ for
any $f\in\Omega$. Since the increasing sequence $\{p_k\}_{k\in\N}$
of seminorms defines the topology $\tau$ and $F$ is
$\tau$-continuous, there exists $k\in\N$ such that $F$ is bounded
with respect to the seminorm $p_k$. The definition of $p_k$ and the
Riesz theorem on the shape of continuous linear functionals on
Banach spaces $C(K)$ imply that there exists a non-zero finite Borel
$\sigma$-additive complex-valued measure $\nu$ on $\T$ supported on
$J_k$ and such that
$$
F(f)=\int_{\T}f(z)\,\nu(dz)=\int_{J_k}f(z)\,\nu(dz)\ \ \text{for any
$f\in C(\T)$.}
$$
Since $\bigl\|\e^{\,c_{|n|}}e_n\bigr\|=1$ for each $n\in\Z$, we see
that $f_n\in\Omega$ for any $n\in\Z$, where
$f_n(z)=\e^{\,c_{|n|}}z^{-n}$. Since $|F(f)|\leq 1$ for any
$f\in\Omega$, we have $|F(f_n)|\leq1$ for $n\in\Z$. Hence
$$
1\geq |F(f_n)|=\e^{\,c_{|n|}}\biggl|\int_\T
z^{-n}\nu(dz)\biggr|=\e^{\,c_{|n|}}|\widehat \nu(n)|\ \ \text{for
any}\ \ n\in\Z.
$$
Thus $|\widehat \nu(n)|\leq \e^{-c_{|n|}}$ for any $n\in\Z$. Since
$c\in\X_+$, we have $\lim\limits_{k\to\infty}
k(c_{k+1}-c_k)=+\infty$. It immediately follows that
$\lim\limits_{k\to\infty}\frac{c_k}{\ln k}=\infty$. Hence
$\e^{-c_k}=o(k^{-j})$ as $k\to\infty$ for every $j\in\N$. Since
$|\widehat \nu(n)|\leq \e^{-c_{|n|}}$ for any $n\in\Z$, we get
$|\widehat \nu(n)|=o(|n|^{-j})$ as $|n|\to\infty$ for each $j\in\N$.
Hence $\nu$ is absolutely continuous with respect to $\mu$ and the
density (=Radon--Nikodym derivative) $\rho=\frac{d\nu}{d\mu}$ is an
infinitely differentiable complex valued function on $\T$. Since
$\widehat\rho(n)=\widehat\nu(n)$ for any $n\in\Z$, we obtain
$|\widehat \rho(n)|\leq \e^{-c_{|n|}}$ for any $n\in\Z$. Since $\nu$
is supported on $J_k$, $\rho$ vanishes on $\T\setminus J_k$. By
Theorem~Mb, $\rho=0$ and therefore $\nu=0$. We have arrived to a
contradiction, which completes the proof of Theorem~\ref{t3}.

\begin{remark}\label{last}\rm We have actually proven a slightly
stronger result than the one stated in Theorem~\ref{t3}. Namely, we
have shown that for each $c\in\X_+$, there is a sequence
$a=\{a_n\}_{n\in\Z}$ of complex numbers such that $|a_n|\leq
\e^{\,c_{|n|}}$ for every $n\in\Z$ and the trigonometric series
$\sum\limits_{n=-\infty}^\infty a_nz^n$ is universal in $C(\T)$
endowed with the metrizable locally convex topology $\tau$ (defined
by the sequence $\{p_k\}_{k\in\N}$ of seminorms from $(\ref{pk})$)
stronger than the pointwise convergence topology.
\end{remark}

\medskip

{\bf Acknowledgements.} \ The author is grateful to the referees for
helpful comments and suggestions.


\end{document}